\documentclass{amsart}

\usepackage{graphicx, amsmath, amssymb, amsthm}
\usepackage{hyperref} 
\usepackage{tikz-cd}

\usepackage{todonotes}
\usepackage{mathrsfs}

\usepackage[backend=biber, style=numeric, maxbibnames=99, doi=true,url=false,giveninits=true, hyperref]{biblatex}
\renewbibmacro{in:}{}

\def\quickop#1{\expandafter\DeclareMathOperator\csname
#1\endcsname{#1}}
\quickop{Orb}\quickop{op}
\newcommand{\aP}{\mathcal{P}}
\newcommand{\sL}{\mathscr{L}}

\newcommand{\NSFSupport}[1]{This material is based upon work supported by the National Science Foundation under Grant No. {#1}} 

\newcommand{\MAHNSF}{DMS-1811189}

\newcommand{\MAHAddress}{University of California Los Angeles, Los Angeles, CA 90095}
\newcommand{\MAHemail}{\tt{mikehill@math.ucla.edu}}

\newcommand{\mycases}[1]{\left\{\begin{array}{ll}#1\end{array}\right.}

\newcommand{\timesover}[1]{\underset{#1}{\times}}
\newcommand{\otimesover}[1]{\underset{#1}{\otimes}}

\DeclareMathOperator{\Map}{Map}

\newcommand{\m}[1]{{\protect\underline{#1}}}

\newcommand{\mM}{\m{M}}

\newcommand{\mR}{\m{R}}

\newcommand{\mN}{\m{N}}

\newcommand{\mA}{\m{A}}

\newcommand{\mSet}{\m{\Set}}

\newcommand{\cc}[1]{\mathcal #1}
\newcommand{\cC}{\cc{C}}
\newcommand{\ccD}{\cc{D}}

\newcommand{\cO}{\cc{O}}

\newcommand{\cP}{\cc{P}}
\newcommand{\cA}{\cc{A}}

\newcommand{\id}{\textrm{id}}

\newcommand{\CoInd}{\textnormal{CoInd}}

\newcommand{\Set}{\mathcal Set}

\newcommand{\Ninfty}{N_\infty}

\newcommand{\Ab}{\mathcal Ab}

\newcommand{\cOrb}{\mathcal Orb}

\newcommand{\Tamb}{\mathcal Tamb}

\newcommand{\Mackey}{\mathcal Mackey}

\newcommand{\SymM}{\mathcal Sym}

\newcommand{\mC}{\m{\cC}}

\mathchardef\mhyphen=45

\newtheorem{theorem}{Theorem}[section]
\newtheorem{lemma}[theorem]{Lemma}
\newtheorem{corollary}[theorem]{Corollary}

\newtheorem{definition}[theorem]{Definition}
\newtheorem{proposition}[theorem]{Proposition}
\newtheorem{conjecture}[theorem]{Conjecture}
\newtheorem{remark}[theorem]{Remark}
\newtheorem{example}[theorem]{Example}

\newtheorem{warning}[theorem]{Warning}

\newtheorem*{definition*}{Definition}
\newtheorem*{motivating}{Main Question}
\newtheorem*{theorem*}{Theorem}
\newtheorem*{proposition*}{Proposition}
\newtheorem*{conjecture*}{Conjecture}
\newtheorem*{corollary*}{Corollary}

\newcommand{\defemph}[1]{\textbf{#1}}

\newcommand{\coinduction}{coinduction}

\newcommand{\bR}{{\mathbb R}}

\title{Bi-incomplete Tambara functors}
\author{Andrew J.~Blumberg}
\address{Department of Mathematics, Columbia University, 
New York, NY \ 10027}
\email{blumberg@math.columbia.edu}

\author{Michael A.~Hill}
\address{\MAHAddress}
\email{\MAHemail}

\thanks{\NSFSupport{DMS-1812064 and {\MAHNSF}}}

\begin{document}

\begin{abstract}
For an equivariant commutative ring spectrum $R$, $\pi_0 R$ has
algebraic structure reflecting the presence of both additive transfers
and multiplicative norms.  The additive structure gives rise to a
Mackey functor and the multiplicative structure yields the additional
structure of a Tambara functor.  If $R$ is an $N_\infty$ ring spectrum
in the category of genuine $G$-spectra, then all possible additive
transfers are present and $\pi_0 R$ has the structure of an incomplete
Tambara functor.  However, if $R$ is an $N_\infty$ ring spectrum in a
category of incomplete $G$-spectra, the situation is more subtle.

In this paper, we study the algebraic theory of Tambara structures on
incomplete Mackey functors, which we call bi-incomplete Tambara
functors.  Just as incomplete Tambara functors have compatibility
conditions that control which systems of norms are possible,
bi-incomplete Tambara functors have algebraic constraints arising from
the possible interactions of transfers and norms.  We give a complete
description of the possible interactions between the additive and
multiplicative structures.
\end{abstract}

\maketitle

\section{Introduction}

The complexity of the equivariant stable category for a finite group
$G$ is a consequence of the desideratum that the orbits $G/H$ must be
dualizable.  In contrast to the non-equivariant setting, there are
many possible variants of the equivariant stable category
corresponding to which orbits are dualizable.  Classically, this
structure is captured by a universe, an infinite-dimensional $G$-inner
product space that contains infinitely many copies of a collection of
finite-dimensional $G$-inner product spaces including $\bR^n$ for each
$n$.  A result of Lewis tells us that $G/H$ is dualizable in the
equivariant stable category structured by $U$ if $G/H$ embeds in
$U$~\cite{LewisChangeofUniverse}.  Another way of saying this is that
the universe controls which transfer maps exist.  On $\pi_0$, a shadow
of this is reflected in the Mackey functor structure.

Equivariant commutative ring spectra have traditionally also been
controlled by a universe in the form of the action of $\sL_G(U)$, the
$G$-equivariant linear isometries operad for a universe
$U$~\cite{LMS}.  Just as the ``additive'' structure of the equivariant
stable category is expressed by  
the presence of transfer maps, the multiplicative structure encoded in
the operad can be described in terms of the multiplicative norms
introduced in Hill--Hopkins--Ravenel~\cite{HHR}.
Building on this, we defined the notion of an
$N_\infty$-operad and showed that these operads control the transfers
and norms in the equivariant stable setting~\cite{BHNinfty}. Moreover, we
showed that the data of these operads is essentially algebraic,
encoded in {\em indexing systems}. An important aspect of this
perspective is that indexing systems capture a more general range of
possible compatible systems of norms or transfers than universes. We
review the definitions and combinatorics of indexing systems (and
associated ``indexing categories'') in Section~\ref{ssec:Indexing}.

This algebra becomes most concrete when we pass to $\pi_0$.  When we
restrict to a model of the equivariant stable category that only has
some transfer maps (e.g., the $\cO$-stable categories
of~\cite{BHOSpectra}), $\pi_0$ has the structure of an 
{\em incomplete Mackey functor}.  That is, for any indexing system we have a
notion of an incomplete Mackey functor associated to that indexing
system.

When working with $N_\infty$ ring spectra, it is standard in the
subject to assume that the additive structure is complete and study
variation in the multiplicative structure.  Then on $\pi_0$ we obtain
various kinds of incomplete Tambara functors, which are Mackey
functors equipped with additional norm maps satisfying certain
compatibilities~\cite{BHOTamb}.  Since incomplete Tambara functors are
less familiar than incomplete Mackey functors, we give a concrete
description.

Tambara functors can be expressed in terms of a particular equivariant
Lawvere theory, being a diagram category  indexed by a category of
``polynomials'' or ``bispans'' (see~\cite{Tambara} and~\cite{Strickland}).

\begin{definition*}[Definition~\ref{def:Poly}]
Let \(\cP^G\) be the category with objects finite \(G\)-sets and where
the morphisms from \(S\) to \(T\) are isomorphism classes of
``polynomials'': diagrams of the form 
\[
T_h\circ N_g\circ R_f:= S\xleftarrow{f}U_1\xrightarrow{g} U_2\xrightarrow{h}T
\]
with \(f\), \(g\), and \(h\) maps in \(\Set^G\).
\end{definition*}
\noindent The composition rules for these are somewhat involved; we review the
theory of polynomials in detail in Section~\ref{ssec:Polynomials}. 

Incomplete Tambara functors are defined by restricting the collection
of maps \(g\), parameterizing the ``norm'' \(N_g\), in the polynomials
to lie in a subcategory of \(\Set^G\). {\emph{A priori}}, this just describes a subgraph of \(\cP^G\);
unpacking the requirements for this subgraph to be a category led us to the definition of an indexing category.  When these maps are in an
indexing category \(\cO\), then this gives the category
\(\cP^G_{\cO}\) of polynomials
with exponents in \(\cO\).

The natural follow-up question is to determine what happens when we vary the ``additive''
structure as well, i.e., restricting the map \(h\) parameterizing the ``transfer'' \(T_h\) to also lie in some
subcategory of \(\Set^G\).  It is clear we at least need to restrict to
considering indexing categories here as well, but additional
compatibility will be required.

\begin{definition*}[Definition~\ref{def:OaOmPoly}]
    Let \(\cO_a\) and \(\cO_m\) be indexing categories, and let \(\cP^G_{\cO_a,\cO_m}\) be the wide directed subgraph of \(\cP^G\) so that the arrows from \(S\) to \(T\) are the isomorphism classes of polynomials
    \[
    T_h\circ N_g\circ R_f:= S\xleftarrow{f}U_1\xrightarrow{g} U_2\xrightarrow{h}T
    \]
    with \(g\in\cO_m\) and \(h\in\cO_a\).
\end{definition*}

In these terms, the following is the main question studied in this paper.

\begin{motivating}
    What compatibility must we have between the additive
    indexing category \(\cO_a\) and the multiplicative indexing category \(\cO_m\) so that the subgraph \(\cP^{G}_{\cO_a,\cO_m}\) of \(\cP^G\) is a subcategory?
\end{motivating}

The subtle point here arises from the ``equivariant distributive property'' which records how to take a norm of a sum or a transfer. Following Mazur, we call this kind of interchange ``Tambara reciprocity'' \cite{MazurThesis}, and a key feature is that the formulae depend only on \(G\) and its subgroups. In general, this will involve transfers and norms connecting many intermediate subgroups.

\begin{example}
    For \(G=C_4\) with generator \(\gamma\), the norm \(N_{e}^{C_4}\) associated to the unique map \(C_4\to\ast\) satisfies
    \begin{multline*}
        N_e^{C_4}(a+b)=
        N_e^{C_4}(a)+ N_e^{C_4}(b)+ tr_{C_2}^{C_4}\big(N_e^{C_2}(a)\cdot\gamma N_e^{C_2}(b)\big) \\
        +tr_e^{C_4}(a\cdot\gamma a\cdot\gamma^2a\cdot\gamma^3b+a\cdot\gamma a\cdot\gamma^2b\cdot\gamma^3b+a\cdot\gamma b\cdot\gamma^2b\cdot\gamma^3 b).
    \end{multline*}
\end{example}

Incomplete Tambara functors are specified by including only a subset
of the possible norm maps.  The
required compatibility check ensures that if we have the norm \(N_K^H\), then we
must also have any of the norms that occur in any Tambara reciprocity
formula.  If we also include only some of the transfers, then we have much
more stringent conditions: as the example shows, we run into several
different transfers in the Tambara reciprocity formulae.

In general, distributivity of the twisted product, parameterized by
some \(g\), over twisted sums is recorded by the ``dependent
product'' \(\Pi_g\) (we review this in Definition~\ref{def:DepProd}
below).  This allows us to concisely state the required compatibility
data.

\begin{definition*}[Definition~\ref{def:Compatible}]
Let $\cO_a$ and $\cO_m$ be indexing categories.  The indexing category
$\cO_m$ {\defemph{distributes over}} $\cO_a$ if for all maps $g\colon
S\to T$ in $\cO_m$, we have
\[
\Pi_g\big((\cO_a)_{/S}\big)\subset(\cO_a)_{/T}.
\]
In this situation, we will say that the pair
\((\cO_a,\cO_m)\) is {\defemph{compatible}}. 
\end{definition*}

The first main theorem, proved in Section~\ref{sec:AddIncomplete}, guarantees that this is the right notion.

\begin{theorem*}[{Theorem~\ref{thm:Compatible}}]
If \((\cO_a,\cO_m)\) is compatible, then \(\cP^{G}_{\cO_a,\cO_m}\) is
a subcategory of \(\cP^G\).
\end{theorem*}

This gives rise to the following basic definition of a bi-incomplete
Tambara functor.

\begin{definition*}[Definition~\ref{def:OaOmTamb}]
Let \((\cO_a,\cO_m)\) be a compatible pair of indexing
subcategories. An \((\cO_a,\cO_m)\)-semi-Tambara functor is a product
preserving functor 
\[
\mR\colon \cP^G_{\cO_{a},\cO_{m}}\to \Set.
\]
An \((\cO_a,\cO_m)\)-Tambara functor is an
\((\cO_a,\cO_m)\)-semi-Tambara functor \(\mR\) such that for all
finite \(G\)-sets \(T\), \(\mR(T)\) is an abelian group.
	
We write \((\cO_a,\cO_m)\mhyphen\Tamb\) to denote the category of
\((\cO_a,\cO_m)\)-Tambara functors.
\end{definition*}

Bi-incomplete Tambara functors have good categorical properties
analogous to the properties of (incomplete) Tambara functors.  We review these in Section~\ref{sec:CatProperties}, including a discussion of the additively incomplete box product. 

In Section~\ref{sec:ReWrite}, we express the notion of compatibility
in more concrete terms, eventually reducing verification of
compatibility to checking a tractable combinatorial condition.

\begin{theorem*}[Theorem~\ref{thm:OaOmCompatibility}]
Let \(\cO_a\) and \(\cO_m\) be indexing categories. Then
\((\cO_a,\cO_m)\) is compatible if and only if for every pair of
subgroups \(K\subset H\) such that \(H/K\) is an admissible \(H\)-set
for \(\cO_m\) and for every admissible \(K\)-set \(T\) for \(\cO_a\),
the coinduced \(H\)-set \(\Map^K(H,T)\) is admissible for \(\cO_a\). 
\end{theorem*}

This formulation makes it clear that there are harsh necessary
conditions on a pair \((\cO_a,\cO_m)\) for them to be compatible; we
explore these in Section~\ref{sec:Limits}.  These conditions alone rule
out about half of all possible pairs!

\begin{proposition*}[Proposition~\ref{prop:OmIsometryLike}]
If \((\cO_a,\cO_m)\) is a compatible pair and \(H/K\) is an
\(\cO_m\)-admissible set for \(H\), then for every \(L\subset H\) such
that \(K\subset L\), the \(H\)-set \(H/L\) is \(\cO_a\)-admissible.  
\end{proposition*}

To get a sense for how this plays out in practice, we analyze the
classical examples coming from equivariant little disks and
equivariant linear isometries operads.  On the one hand, we find that
the little disks do not necessarily interact well with each other:

\begin{corollary*}[Corollary~\ref{cor:SelfIncompatible}]
For any non-simple group \(G\), there is a universe \(U\) such that
the indexing category associated to the little disks in \(U\) is not
compatible with itself. 
\end{corollary*}

On the other hand, as is implicit in the classical literature, the
indexing category for the linear isometries operad on a universe \(U\)
is always compatible with the indexing category corresponding to the
little disks operad for \(U\).

\begin{proposition*}[Proposition~\ref{prop:DisksIsomWorks}]
Let $U$ be a universe for $G$, let $\cO_a$ be the indexing category
associated to the little disks operad for $U$ and let $\cO_m$ be the
indexing category associated to the linear isometries operad for
$U$. Then \((\cO_a,\cO_m)\) is compatible. 
\end{proposition*}

We close in Section~\ref{sec:Dragons} by proving some basic
change-of-group results and then putting forward a series of
conjectures about an ``external'' form of bi-incomplete Tambara
functors.  The thesis work of Mazur and of Hoyer~\cite{MazurThesis,
  HoyerThesis} showed that Tambara functors were Mackey functors with
additional structure, i.e., external norm maps.  Our conjectures
outline how such a description should work for bi-incomplete Tambara
functors, exhibiting them as ordinary incomplete Mackey functors
together with additional structure.

A key step is producing additively incomplete versions of the norm
functor.  

\begin{conjecture*}[Conjecture~\ref{conj:NormsCompat}]
    If \((\cO_a,\cO_m)\) is a compatible pair of indexing categories, and \(H/K\) is a \(\cO_m\)-admissible \(H\)-set, then there is a norm functor
    \[
        N_K^H\colon i_K^\ast\cO_a\mhyphen\Mackey\to i_H^\ast\cO_a\mhyphen\Mackey
    \]
    that is symmetric monoidal with respect to the box product.
\end{conjecture*}

These would assemble into the incomplete Mackey functor version of the
multiplicative symmetric monoidal Mackey functor structure on the
\(G\)-Mackey functors. In particular, we would hope to have the
analogue of the Hoyer--Mazur theorem that Tambara functors are
\(G\)-commutative monoids in Mackey functors. 

\begin{conjecture*}[Conjecture~\ref{conj:OmCommMonoids}]
For any compatible pair of indexing categories \((\cO_a,\cO_m)\), there is an equivalence of
categories between \(\cO_m\)-commutative monoids in \(\cO_a\)-Mackey
functors and \((\cO_a,\cO_m)\)-Tambara functors. 
\end{conjecture*}

\subsection*{Acknowledgements}

We offer our sincere, heartfelt thanks to John Greenlees.  John
offered his generous support to both of us as young people in the
field, and we have found his beautiful mathematics deeply influential.
In addition, the myriad ways he has contributed to the homotopy theory
community has been an inspiration.

We thank Magdalena Kedziorek for her support and understanding through
this project, and we also thank the referee for an incredibly fast and detailed report.

Finally, we thank Mike Hopkins, Tyler Lawson, Mike Mandell, Peter May,
and Jonathan Rubin for many interesting and helpful conversations
about this and related mathematics. 

\section{Indexing systems, subcategories, and polynomials}

The purpose of this section is to review the basic algebraic framework
for incomplete Mackey and Tambara functors.  Throughout the section,
we will fix a finite group $G$ and denote by \(\Set^G\) the category
of finite $G$-sets.

\subsection{Indexing systems and categories}\label{ssec:Indexing}

We begin with the definition of an indexing system; this is the basic
categorical structure that organizes the possible relationships
between transfers for different finite $G$-sets.

\begin{definition}
A \defemph{symmetric monoidal coefficient system} is a functor 
\[
\mC\colon\cOrb_{G}^{op}\to\SymM
\]
from the opposite of the orbit category of $G$ to the category of symmetric monoidal categories and strong symmetric monoidal functors.
\end{definition}

\begin{example}
The fundamental example of a symmetric monoidal coefficient
system is the functor which assigns to $G/H$ the category of finite
$H$-sets, with the symmetric monoidal structure induced by disjoint
union; we denote this by $\mSet^{\amalg}$.
\end{example}

\begin{example}
    There is also a multiplicative version which assigns to $G/H$ the category of finite $H$-sets with the Cartesian symmetric monoidal structure; we will denote this by \(\mSet^{\times}\)
\end{example}

\begin{definition}\label{defn:IndexingSystem}
An {\defemph{indexing system}} is a full symmetric monoidal
sub-coefficient system $\cO$ of $\mSet^{\amalg}$ that contains all
trivial sets and is closed under 
\begin{enumerate}
\item levelwise finite limits and
\item ``self-induction'': if $H/K\in \cO(G/H)$ and $T\in\cO(G/K)$, then 
\[
H\timesover{K}T\in\cO(G/H).
\]
\end{enumerate}
\end{definition}

The collection of indexing systems for $G$ forms a poset ordered by
inclusion.  This poset has a least and a greatest element.

\begin{example}\label{exam:Otr}
The poset of indexing systems has a least element: \(\cO^{tr}\), for
which \(\cO^{tr}(G/H)\) is always the subcategory of \(\Set^H\) of
\(H\)-sets with a trivial \(H\)-action. 
\end{example}

\begin{example}
The symmetric monoidal coefficient system \(\mSet^{\amalg}\) itself is
an indexing system, \(\cO^{gen}\), and this is the maximal element in
the poset of indexing systems.
\end{example}

We will often write $\cO(H)$ for $\cO(G/H)$. 

\begin{definition}
We say that an \(H\)-set \(T\) is admissible if \(T\in\cO(H)\).
\end{definition}

Indexing systems admit an intrinsic formulation, which we can interpret
equivalently as gluing together all of \(\cO(H)\), viewed as
subcategories of slice categories of \(\Set^G\) via the equivalences
\[
    \Set^H\xrightarrow[\simeq]{G\timesover{H}(\mhyphen)}\Set^G_{/(G/H)}
\]

\begin{definition}
An {\defemph{indexing category}} is a wide, pullback stable, finite
coproduct complete subcategory \(\cO\) of the category of finite
\(G\)-sets. 
\end{definition}

Indexing categories also form a poset under inclusion, and pullback stability and
finite coproduct completeness guarantee that the assignment 
\[
G/H\mapsto \cO_{/(G/H)}
\]
gives us a map of posets from the poset of indexing categories to the poset of
indexing systems.  

\begin{theorem}[{\cite[Theorem 1.4]{BHOTamb}}]
The map from the poset of indexing categories to that of indexing
systems is an isomorphism. 
\end{theorem}

Because of this isomorphism, we will engage in mild abuse of notation
and use the same symbols to denote both indexing categories and the
associated indexing systems.

\subsection{Polynomials}\label{ssec:Polynomials}

The point of indexing categories, as opposed to indexing systems, is
that this reformulation provides a convenient formalism for 
parameterizing norms for incomplete Tambara functors~\cite{BHOTamb}
and for incomplete Mackey functors~\cite{BHOSpectra}.  
Specifically, we employ a categorification of the notion of
polynomials (also called ``bispans'' \cite{Strickland}).  Although this
definition works in the context of locally cartesian closed categories 
(by work of Gambino--Kock \cite{GamKock} and Weber \cite{Weber}), we
focus here for concreteness on \(\Set^G\).

\begin{definition}\label{def:Poly}
The category $\aP^{G}$ of {\em polynomials} has objects finite $G$-sets and
morphisms between objects $S$ and $T$ the isomorphism classes of
polynomials 
\[
T_h\circ N_g\circ R_f:= S\xleftarrow{f}U_1\xrightarrow{g} U_2\xrightarrow{h}T
\]
where an isomorphism of polynomials is specified by isomorphisms $U_1 \to U_1'$ and $U_2 \to U_2'$ that make the evident diagrams commute.  
\end{definition}

The composition in the category is a little elaborate; we spend the rest of this subsection unpacking and explaining it.

\begin{remark}
    The category $\aP^{G}$ is obtained from a $2$-category where the category of morphisms is the category with objects polynomials and morphisms maps of polynomials as in the definition, where the internally described square is a pullback square.
\end{remark}

The category of polynomials can be presented in terms of generators
and relations, which is often technically convenient.  Specifically,
as indicated, \(T_h\circ N_g\circ R_f\) is a composite of basic maps:

\begin{definition}
Let $f \colon S \to T$ be a map of finite $G$-sets.  Then we define the
following morphisms
\begin{align*}
R_f &= T \xleftarrow{f} S \xrightarrow{\id}  S \xrightarrow{\id}  S \\
N_f &= S \xleftarrow{\id} S \xrightarrow{f}  T \xrightarrow{\id}  T \\
T_f &= S \xleftarrow{\id} S \xrightarrow{\id} S \xrightarrow{f} T.
\end{align*}
\end{definition}

These are stand-ins for the three basic ways we might build the
polynomials on a sequence of symbols: 
\begin{enumerate}
\item[R] Repeat or drop some collection of the variables, then
\item[N] multiply collections of them (with ``N'' for the Galois theoretic norm), and
\item[T] sum up the result (with ``T'' for the transfer or trace).
\end{enumerate}
Heuristically, the map \(N\) should be thought of as ``multiply
together the fibers over a point'' and \(T\) as ``sum together the
fibers''.  

Composition in this category is most easily expressed by showing how
to transform an arbitrary string of composable combinations of
\(T\), \(N\), and \(R\)s into one of the form in Definition~\ref{def:Poly}.

First, we consider composing \(T\)s with \(T\)s, etc.  In terms of our
heuristic, we can duplicate, multiply, or add things either in stages
or all at once.

\begin{proposition}
For any composable maps \(f\colon S\to T\) and \(g\colon T\to U\), we
have 
\begin{align*}
R_{g\circ f}&=R_f\circ R_g\\
T_{g\circ f}&=T_g\circ T_f\\
N_{g\circ f}&=N_g\circ N_f.
\end{align*}
\end{proposition}
\noindent In other words, \(T\) and  \(N\) extend to covariant functors \(\Set^G\to \cP^G\), while \(R\) extends to a contravariant one.

Next, we encode the heuristic that adding or multiplying and then
duplicating is the same as first duplicating and then adding or
multiplying coordinatewise.  This is expressed in terms of pullbacks,
as follows.

\begin{proposition}
Given a pullback diagram
\[
\begin{tikzcd}
{T'} \ar[r, "{g'}"] \ar[d, "{f'}"'] & {T}\ar[d, "f"]  \\
{S'} \ar[r, "g"'] & {S,}
\end{tikzcd}
\]
we have identities 
\[
R_f\circ T_g=T_{g'}\circ R_{f'}\text{ and }R_f\circ N_g=N_{g'}\circ 
R_{f'}. 
\]
\end{proposition}

The most complicated interchange is swapping \(T\) and \(N\), which is
a form of generalized distributivity. This uses the dependent product, which is heuristically the ``product over the fibers''.

\begin{definition}\label{def:DepProd}
If \(g\colon S\to T\) is a map of finite \(G\)-sets, the
{\defemph{dependent product along \(g\)}} is the functor 
\[
\Pi_g\colon\Set^G_{/S}\to\Set^G_{/T}
\]
that is the right adjoint to the pullback along \(g\).
\end{definition}

Recall that an exponential diagram is a
diagram isomorphic to one of the form 
\[
    \begin{tikzcd}
        {T}
            \ar[d, "g"']
            &
        {S}
            \ar[l, "h"']
            &
        {T\timesover{U}\Pi_g(S)}
            \ar[l, "{f'}"']
            \ar[d, "{g'}"]
            \\
        {U}
            &
            &
        {\Pi_g(S),}
            \ar[ll, "{h'}"]
    \end{tikzcd}
\]
where 
\begin{enumerate}
\item \(h'\) is dependent product of \(h\) along \(g\), 
\item \(g'\) is the pullback of \(g\) along \(h'\), and 
\item where \(f'\) is the counit of the pullback-dependent product
  adjunction.
\end{enumerate}  
Notice that only the maps \(h\) and \(g\) can vary freely; all of the
other pieces are pseudo-functorially determined.

\begin{proposition}
    If we have an exponential diagram
    \[
    \begin{tikzcd}
        {T}
            \ar[d, "g"']
            &
        {S}
            \ar[l, "h"']
            &
        {T\timesover{U}\Pi_g(S)}
            \ar[l, "{f'}"']
            \ar[d, "{g'}"]
            \\
        {U}
            &
            &
        {\Pi_g(S),}
            \ar[ll, "{h'}"]
    \end{tikzcd}
    \]
    then
    \[
    N_g\circ T_h=T_{h'}\circ N_{g'}\circ R_{f'}.
    \]
\end{proposition}
Put another way, the exponential diagram records concisely how to
express a general formula for the product of a sum.  

Finally, we record an unexpected categorical result that arises from
the asymmetry in the roles of \(R\), \(S\), and \(T\). 

\begin{proposition}
For any objects \(S\) and \(T\), the maps
\[
\pi_S=[S\amalg T\leftarrow S\xrightarrow{=} S\xrightarrow{=} S]\text{
  and } 
\pi_T=[S\amalg T\leftarrow T\xrightarrow{=} T\xrightarrow{=} T]
\]
are the projection maps witnessing \(S\amalg T\) as the categorical
product in \(\cP^G\). 
\end{proposition}

In other words, regarding \(R\) as a functor
\[
R\colon \Set^{G,op}\to\cP^G,
\]
it is product preserving.

\subsection{Incomplete Tambara functors}

A natural question to ask is when the subset of polynomials with maps
in a restricted subcategory of finite $G$-sets itself forms a
category.  The key observation is that this occurs precisely when the
subcategory in question is an indexing category~\cite[2.10]{BHOTamb}.

\begin{theorem}
Let $\cO$ be an indexing category.  The category $\aP_\cO^{G}$ of {\em
  polynomials with exponents in $\cO$} has objects finite $G$-sets and
morphisms between objects $S$ and $T$ the isomorphism classes of bispans 
\[
S\xleftarrow{f}U_1\xrightarrow{g} U_2\xrightarrow{h}T,
\]
where $g$ is an arrow in $\cO$. Disjoint union of finite \(G\)-sets is the categorical product.
\end{theorem}

With this definition in hand, we can define incomplete Tambara
functors as follows.

\begin{definition}[{\cite[Definition 4.1]{BHOTamb}}]
Let $\cO$ be an indexing category.  An {\em $\cO$-semi-Tambara functor} is
a product-preserving functor
\[
\aP_\cO^{G} \to \Set.
\]
An {\em $\cO$-Tambara functor} is an $\cO$-semi-Tambara functor which is
valued in abelian groups.
\end{definition}

\subsection{Incomplete Mackey functors}

Tambara's category of polynomials is the multiplicative generalization
of Lindner's category of spans as a way to record ``linear
functions'' \cite{Lindner}.  When we are interested in just the additive structure,
we can choose the map $g$ to be the identity and then we recover the
Mackey version of the incomplete polynomials from~\cite{BHOSpectra}.  

\begin{definition}[{\cite[Definition 2.23]{BHOSpectra}}]
Let \(\cO\) be an indexing category.  Then \(\cA^G_{\cO}\) is the
category with objects finite \(G\)-sets and morphisms isomorphism
classes of spans 
\[
\cA^{G}_{\cO}(S,T)=\Big\{[S\leftarrow U\xrightarrow{h} T]\mid h\in
   {\cO}\Big\}, 
\]
with composition given by pullback.
\end{definition}

\begin{remark}
    Once again, this is the quotient category associated to a $2$-category in which we keep track of the coherence of iterated pullbacks.
\end{remark}

\begin{example}\label{exam:BurnsideCat}
    When \(\cO=\Set^G\), the category \(\cA^G_{\cO}\) is the usual Lindner category of spans \(\cA^G\).
\end{example}

In the case of \(\cO=\Set^G\), this category is isomorphic to the
subcategory of polynomials where the norm maps are only along
isomorphisms, and this showed how to extract the underlying additive
Mackey functor of an incomplete Tambara functor.

The disjoint union is again the product in \(\cA^{G}_{\cO}\) (and in
fact, here it is the biproduct). 

\begin{definition}[{\cite[Definition 2.24]{BHOSpectra}}]
An \(\cO\)-Mackey functor is a product preserving functor
\[
\mM\colon\cA^G_{\cO}\to \Set.
\]
A semi-Mackey functor is a Mackey functor if it is group complete.
  
A map of \(\cO\)-Mackey functors is a natural
transformation. We will denote the category of \(\cO\)-Mackey functors
by $\cO\mhyphen\Mackey$.
\end{definition}

\begin{definition}
The group completion of the representable \(\cA_{\cO}(\ast,\mhyphen)\)
is the \(\cO\)-Burnside Mackey functor: \(\mA_{\cO}\). For any subgroup \(H\subset G\), \[
    \mA_{\cO}(G/H)=\mathbb Z\big\{[H/K]\mid H/K\in\pi_0\cO(H)\big\}
\]
is the group completion of the commutative monoid \(\pi_0\cO(H)\) of isomorphism classes of admissible \(H\)-sets.
\end{definition}

\section{Additive Incompleteness}\label{sec:AddIncomplete}

We now study the common generalization of incomplete Mackey and
Tambara functors in terms of polynomials with additional
restrictions.

\begin{definition}\label{def:OaOmPoly}
If $\cO_{a}$ and $\cO_m$ are indexing categories, then let $\cP_{\cO_a,\cO_m}^{G}$ by the wide subgraph of $\cP^{G}$ with morphisms the isomorphism classes of polynomials
\[
X\leftarrow S\xrightarrow{g} T\xrightarrow{h} Y,
\]
with $g\in \cO_m$ and $h\in\cO_a$.
\end{definition}

These are polynomials with norms parameterized by \(\cO_m\) and transfers parameterized by \(\cO_a\). We have no reason to believe, {\emph{a priori}}, that this subgraph is actually a subcategory, and we can see what can go wrong with an explicit example. 
\begin{example}
    A \(G=C_2\)-Tambara functor is the following data:
\begin{enumerate}
    \item A commutative Green functor \(\mR\) and
    \item a multiplicative map \(n\colon \mR(C_2/\{e\})\to \mR(C_2/C_2)\)
\end{enumerate}
that are required to satisfy the relations
\begin{enumerate}
    \item If \(\bar{x}\) is the Weyl conjugate of \(x\) in \(\mR(C_2/\{e\})\), then
    \[
    n(x)=n(\bar{x})\text{ and } Res\circ n(x)=x\cdot\bar{x},
    \]
    where \(Res\) is the restriction map \(\mR(C_2/C_2)\to \mR(C_2/\{e\})\), and
    \item for any \(a,b\in\mR(C_2/\{e\})\), we have
    \[
    n(a+b)=n(a)+n(b)+Tr(a\cdot\bar{b}),
    \]
    where \(Tr\) is the transfer map \(\mR(C_2/\{e\})\to\mR(C_2/C_2)\).
\end{enumerate}
In particular, the existence of the norm in this case necessitates the existence of the transfer map.
\end{example}

\begin{remark}
When dealing with specific classes of groups $G$, it might be
reasonable to consider weaker compatibility requirements than the ones
we study here.  Specifically, it is not the case that for every group
\(G\) we need to require transfers for all elements to consistently
talk about norms. In fact, Georgakopoulos has shown that in that
absence of torsion, the transfers themselves are forced by the
existence of certain norms~\cite{Georgakopoulos}.
\end{remark}

This puts some constraints on the additive indexing category. In fact, compatibility is a purely combinatorial condition. 

\begin{definition}\label{def:Compatible}
Let $\cO_a$ and $\cO_m$ be indexing categories. The indexing system $\cO_m$ {\defemph{distributes over}} $\cO_a$ if for all maps $h\colon S\to T$ in $\cO_m$, we have
\[
\Pi_h\big((\cO_a)_{/S}\big)\subset(\cO_a)_{/T}.
\]
In this situation, we will say that the pair \((\cO_a,\cO_m)\) is {\defemph{compatible}}.
\end{definition}

The force of this definition is given by the following theorem.

\begin{theorem}\label{thm:Compatible}
	If \((\cO_a,\cO_m)\) is compatible, then \(\cP^{G}_{\cO_a,\cO_m}\) is a subcategory of \(\cP^G\).
\end{theorem}

\begin{proof}
We need to verify the interchange formulae for \(T_h\), \(N_g\), and \(R_f\), where \(h\in\cO_a\), \(g\in\cO_m\), and \(f\) is arbitrary. 
	
	For interchanges with \(R\), assume we have pullback squares
	\[
		\begin{tikzcd}
			{S'}
				\ar[r, "{f'}"]
				\ar[d, "{g'}"']
				&
			{S}
				\ar[d, "g"]
				\\
			{T'}
				\ar[r, "f"']
				&
			{T,}
		\end{tikzcd}\qquad\text{ and }\qquad
		\begin{tikzcd}
			{U'}
				\ar[r, "{f''}"]
				\ar[d, "{h'}"']
				&
			{U}
				\ar[d, "h"]
				\\
			{T'}
				\ar[r, "f"']
				&
			{T,}
		\end{tikzcd}
	\]
	where \(g\in\cO_m\) and \(h\in\cO_a\). By pullback stability, \(g'\) is also in \(\cO_m\) and \(h'\) in \(\cO_a\), and hence 
	\[
		R_f\circ N_g=N_{g'}\circ R_{f'}\text{ and } R_f\circ T_h=T_{h'}\circ R_{f''}
	\]
	are again elements of \(\cP^G_{\cO_a,\cO_m}\). 
	
	Since both are subcategories, the compositions of \(T\)s or \(N\)s are also correct: for any composable pair \(g_1,g_2\in\cO_m\) and \(h_1,h_2\in\cO_a\),
	\[
	N_{g_1}\circ N_{g_2}=N_{g_1\circ g_2}\text{ and }T_{h_1}\circ T_{h_2}=T_{h_1\circ h_2}
	\]
	is of the desired form. 
	
	The key thing to check is therefore the interchange between \(T_h\) and \(N_g\). For this, let
	\[
	\begin{tikzcd}
		{T}
			\ar[d, "g"']
			&
		{S}
			\ar[l, "h"']
			&
		{T\timesover{U} \Pi_{h}(g)}
			\ar[l, "{f'}"']
			\ar[d, "{g'}"]
			\\
		{U}
			&
			&
		{\Pi_{h}(S)}
			\ar[ll, "{h'}"]
	\end{tikzcd}
	\]
	be an exponential diagram expressing the interchange relation
	\[
		N_g\circ T_h=T_{h'}\circ N_{g'}\circ R_{f'}.
	\]
	The map \(g'\) is the pullback of \(g\) along \(h'\), and hence \(g\) being in \(\cO_m\) means that \(g'\) is in \(\cO_m\). The assumption that \((\cO_a,\cO_m)\) is compatible means exactly that \(h'\) is in \(\cO_a\), and hence the interchange is again in \(\cP^G_{\cO_a,\cO_m}\).
	\end{proof}
	
\begin{example}
    For any \(\cO_m\), the pair \((\cO^{gen},\cO_m)\) is compatible.
    The associated category of polynomials are precisely the
    ``polynomials with exponents in \(\cO\)'' studied in~\cite{BHOTamb}.
\end{example}

Since the categories \(\cO_a\) and \(\cO_m\) are indexing categories, for any objects \(S\) and \(T\), the maps
\[
\pi_S=[S\amalg T\leftarrow S\xrightarrow{=} S\xrightarrow{=} S]\text{ and }
\pi_T=[S\amalg T\leftarrow T\xrightarrow{=} T\xrightarrow{=} T]
\]
are always in \(\cP^{G}_{\cO_a,\cO_m}\). These are the projection maps in the category \(\cP^{G}\), so we deduce the following.

\begin{proposition}\label{prop:ProductsExist}
	If \((\cO_a,\cO_m)\) is compatible, then the disjoint union of finite \(G\)-sets is the categorical product in \(\cP^G_{\cO_a,\cO_m}\).
\end{proposition}

In ordinary Mackey or incomplete Tambara functors, additive or
multiplicative monoid structures on the mapping sets arise from
transfering or norming along the fold maps. The proof of the usual
case for Tambara functors (as in~\cite{Tambara} or also
in~\cite{Strickland}) goes through without change in the bi-incomplete
case, since the key features are that all of the terms are in fact
again in the category in question.  That is, in the bi-incomplete
case, we simply skipped some transfers as well as norms.  

\begin{proposition}\label{prop:MonoidStructure}
Let \((\cO_a,\cO_m)\) be a compatible pair of indexing
subcategories. Then the hom objects in \(\cP^{G}_{\cO_a,\cO_m}\) are
naturally commutative semi-ring valued, with addition given by 
    \[
            [S\leftarrow T_1\to T_2\to U]+
            [S\leftarrow V_1\to V_2\to U]=
            [S\leftarrow T_1\amalg V_1\to T_2\amalg V_2\to U],
    \]
    and multiplication given by
    \begin{multline*}
        [S\leftarrow T_1\to T_2\to U]\cdot
        [S\leftarrow V_1\to V_2\to U]=\\
        \Big[S\leftarrow \big((T_1\timesover{U} V_2)\amalg (T_2\timesover{U} V_1)\to T_2\timesover{U} V_2\to U\Big].
    \end{multline*}
    
    The transfer maps are map of additive monoids and the norm maps are maps of multiplicative monoids.
\end{proposition}

This is formally exactly like what we see with Tambara and incomplete
Tambara functor.  There the hom objects are commutative semi-rings, and
the restriction maps are ring homomorphisms.  

\begin{definition}\label{def:OaOmTamb}
	Let \((\cO_a,\cO_m)\) be a compatible pair of indexing categories. An \((\cO_a,\cO_m)\)-semi-Tambara functor is a product preserving functor
	\[
		\mR\colon \cP^G_{\cO_{a},\cO_{m}}\to \Set.
	\]
	An \((\cO_a,\cO_m)\)-Tambara functor is an \((\cO_a,\cO_m)\)-semi-Tambara functor \(\mR\) such that for all finite \(G\)-sets \(T\), \(\mR(T)\) is an abelian group.
	
	A map of \((\cO_a,\cO_m)\)-Tambara functors is a natural transformation of product preserving functors.
	
	We will write \((\cO_a,\cO_m)\mhyphen\Tamb\) to denote the
        category of \((\cO_a,\cO_m)\)-Tambara functors. 
\end{definition}

Implicit in the definition is that any product preserving functor 
\[
    \cP^G_{\cO_a,\cO_m}\to\Set
\]
is naturally commutative semi-ring valued. This was a fundamental result of Tambara in the ordinary Tambara functor case \cite{Tambara}. Tambara also showed that the group completion of any semi-Tambara functor naturally has the structure of a Tambara functor. As in the incomplete Tambara functor case, the proof goes through without change.

\begin{proposition}[{\cite[]{Tambara}, \cite[]{Strickland}}]
    Given any \((\cO_a,\cO_m)\)-semi-Tambara functor \(\mR\), the object-wise group completion is an \((\cO_a,\cO_m)\)-Tambara functor.
\end{proposition}

\subsection*{More general compatibility}

The proof of Theorem~\ref{thm:Compatible} only used that indexing
categories are wide, pullback stable, together with a condition of
closure under dependent product.  That is, the same proof works to
establish the analogous result in this more general case.  It turns out that
the wide subcategory of isomorphisms \(\Set^G_{\cong}\), which is not
an indexing category, will be necessary to define the ``underlying''
incomplete additive and multiplicative Mackey functors associated to
a bi-incomplete Tambara functor.

\begin{proposition}\label{prop:Isos}
    Let \(\ccD\) be any wide, pullback stable subcategory. Then
    \begin{enumerate}
        \item \(\Set^G_{Iso}\) distributes over \(\ccD\), and
        \item \(\ccD\) distributes over \(\Set^G_{Iso}\).
    \end{enumerate}
\end{proposition}

\begin{proof}
For the first claim, note that the pullback along an isomorphism is an
equivalence of categories, and hence the right adjoint, the dependent product, is naturally
isomorphic to the pullback along the inverse to the isomorphism.  Hence
any pullback stable subcategory is closed under the dependent product along an isomorphism. 
    
For the second claim, since dependent product is a functor, it takes
isomorphisms to isomorphisms.
\end{proof}

Note that $\Set^G_{Iso}$ is initial amongst all wide,
pullback stable subcategories of \(\Set^G\).

\section{Categorical Properties}\label{sec:CatProperties}

In this section, we describe the formal structure of the category of
bi-incomplete Tambara functors.

\subsection{The box product for incomplete Mackey functors}

The box product of Mackey functors is essential in showing that
Tambara functors have colimits, since this serves as a model for the
Mackey functor underlying the coproduct.  For coproducts in the
bi-incomplete case, we need to build the box product of incomplete
Mackey functors.

\begin{lemma}\label{lem:ClosureUnderProducts}
    Indexing categories are closed under products: if \(f_i\colon S_i\to T_i\) are in \(\cO\) with \(i=1,2\), then
    \[
    f_1\times f_2\colon S_1\times S_2\to T_1\times T_2
    \]
    is as well.
\end{lemma}
\begin{proof}
    It suffices to show this when \(f_1\) is the identity, since 
    \[
    f_1\times f_2=(1\times f_2)\circ (f_1\times 1),
    \]
    and the twist map is an isomorphism (and hence in any indexing category). Since the Cartesian product distributes over the disjoint union and since indexing categories are pullback stable, we reduce to the case \(S_1=G/H\). This means we consider 
    \[
    1 \times f\colon G/H\times S\to G/H\times T.
    \]
    We have a commutative square
    \[
        \begin{tikzcd}
        {G/H\times S}
            \ar[d, "{1\times f}"']
            &
        {G\timesover{H}i_H^\ast S}
            \ar[l, "\cong"']
            \ar[d, "{1\timesover{h} i_H^\ast f}"]
            \\
        {G/H\times T}   
            \ar[r, "\cong"']
            &
        {G\timesover{H}i_H^\ast T,}
        \end{tikzcd}
    \]
    where the unlabeled maps are the natural ``shearing'' isomorphisms (the asymmetrical directions of the horizontal maps is to help make transparent the form of the right-hand map). However, by \cite[Proposition 3.13]{BHOTamb}, induction preserves indexing categories, and by \cite[Proposition 6.3]{BHOTamb}, so does restriction.
\end{proof}

The Cartesian product endows \(\cA^G\) with a symmetric monoidal
structure (e.g., see the discussion in~\cite[Section 3]{Strickland}).
Lemma~\ref{lem:ClosureUnderProducts} shows this is compatible with the inclusions
\[
\cA^G_{\cO}\hookrightarrow \cA^G
\]
induced by \(\cO\hookrightarrow \Set^G\), where \(\cA^G\) is the Lindner category of Example~\ref{exam:BurnsideCat}.

\begin{corollary}\label{cor:AOSymMonoid}
The category \(\cA^G_{\cO}\) is a symmetric monoidal subcategory of
\(\cA^G\) under the Cartesian product.
\end{corollary}

The box product on Mackey functors is the Day convolution product of
the Cartesian product in \(\cA^{G}\) with the tensor product on \(\Ab\)
\cite{Day}.  We make an analogous definition for incomplete Mackey
functors.

\begin{definition}
    The box product on \(\cO\)-Mackey functors is the left Kan extension of the tensor product along the Cartesian product: given \(\mM\) and \(\mN\), the box product is defined by
    \[
        \begin{tikzcd}
            {\cA^G_{\cO}\times\cA^G_{\cO}}
                \ar[d, "\times"']
                \ar[r, "{\mM\otimes\mN}"]
                &
            {\Ab}
                \\
            {\cA^G_{\cO}}
                \ar[ur, "{\mM\Box\mN}"']
        \end{tikzcd}
    \]
\end{definition}

\begin{remark}
Strickland shows a very important generalization of this: we
can actually take the left Kan extension not of the tensor product to
\(\Ab\) but rather the Cartesian product to \(\Set\). 
    
This is essential for Tambara functors: Any product
preserving functor from \(\cP^G\) to \(\Ab\) is necessarily zero,
since a multiplication that is both linear and bilinear must be zero.
\end{remark}

The definition as the left Kan extension gives a universal property of this incomplete box product completely analogous to the usual one for the box product: a map of \(\cO\)-Mackey functors
\[
\mM\Box\mN\to\m{P}
\]
is the same data as a natural transformation of functors on \(\cA^{\times 2}_{\cO}\)
\[
\mM(\mhyphen)\otimes\mN(\mhyphen)\rightarrow\m{P}(\mhyphen\times\mhyphen).
\]
We note there is a slight wrinkle with the usual Frobenius relation:
the equations
\[
a\otimes T_f(b)=T_f\big(R_f(a)\otimes b\big)\text{ and }T_f(b)\otimes a=T_f\big(b\otimes R_f(a)\big)
\]
only make sense when \(f\) is a map in \(\cO\). When \(\cO=\cO^{tr}\), then any map preserves isotropy and the Frobenius relation just expresses that the tensor is bilinear.
\begin{example}
    If \(\cO=\cO^{tr}\) is the initial indexing category of Example~\ref{exam:Otr}, then the box product is just the ordinary levelwise tensor product. 
\end{example}

General properties of the Day convolution product show that this is a symmetric monoidal product on Mackey functors here too. 
\begin{proposition}
    The box product is a symmetric monoidal product on \(\cO\)-Mackey functors with unit the \(\cO\)-Burnside Mackey functor.
\end{proposition}

\begin{definition}
    An \(\cO\)-Green functor is a commutative monoid for the box product in \(\cO\)-Mackey functors.
\end{definition}

\subsection{Colimits and limits}
Strickland's careful treatment of the complete Tambara case actually goes through without change! At no point does Strickland use that we have transfers and norms for all maps in \(\Set^G\), using instead the compatibility relations needed for particular given maps.

For coproducts, he uses that we can formally create the ``norm of a
transfer'' in the box product by using the corresponding exponential
diagram.  We have to check that in both cases, we are working entirely in our restricted subcategory \(\cP^G_{\cO_a,\cO_m}\).

\begin{proposition}[{\cite[Lemma 9.8]{Strickland}}]
    The coproduct of \((\cO_a,\cO_m)\)-Tambara functors is the box product for \(\cO_a\)-Mackey functors.
\end{proposition}

For coequalizers, Strickland works very generally with relations in a
wide collection of algebraic structures (\cite[Definition
  10.4]{Strickland}).  

\begin{proposition}[{\cite[Proposition 10.5]{Strickland}}]
Coequalizers exist in \((\cO_a,\cO_m)\)-Tambara functors.
\end{proposition}

Since filtered colimits commute with finite products in \(\Set\) and since
\((\cO_a,\cO_m)\)-Tambara functors are a full subcategory of a diagram
category defined in terms of a product condition, we can deduce the
existence of filtered colimits~\cite[Proposition 10.2]{Strickland}.
Putting this all together, we conclude that all colimits exist.

\begin{theorem}
The category of \((\cO_a,\cO_m)\)-Tambara functors is cocomplete.
\end{theorem}

Since the categories of bi-incomplete semi-Tambara functors are diagram categories in a complete category, they are automatically complete.

\begin{theorem}
For any compatible pair \((\cO_a,\cO_m)\), the category of
\((\cO_a,\cO_m)\)-Tambara functors is complete, with limits formed
objectwise. 
\end{theorem}

\subsection{Forgetful Functors}

Inclusions of indexing categories naturally give rise to inclusions of
the corresponding wide subgraphs of polynomials. We only want to
consider the case where the pair of indexing categories is compatible, so we
focus on inclusions here. 

\begin{definition}
    Let \((\cO_a,\cO_m)\) and \((\cO_a',\cO_m')\) be two compatible pairs of indexing categories. We will write
    \[
        (\cO_a,\cO_m)\subseteq (\cO_a',\cO_m')
    \]
    and say that we have an inclusion of pairs
    if we have (not necessarily proper) inclusions 
    \[
        \cO_a\subseteq \cO_a'\text{ and }\cO_m\subseteq\cO_m'.
    \]
\end{definition}

The following is immediate from the definitions.

\begin{proposition}
	If \((\cO_a,\cO_m)\subseteq (\cO_a',\cO_m')\) is an inclusion of compatible pairs of indexing categories, then the natural inclusion
	\[
		\cP^{G}_{\cO_a,\cO_m}\hookrightarrow \cP^{G}_{\cO_a',\cO_m'}
	\]
	is a product-preserving functor.
\end{proposition}

These inclusions give us ``forgetful functors'', since the composite
of product preserving functors is product preserving:

\begin{proposition}
If \((\cO_a,\cO_m)\subseteq (\cO_a',\cO_m')\) is an inclusion of compatible pairs of indexing categories, then precomposition with the inclusion of polynomials gives a forgetful functor
    \[
        (\cO_a',\cO_m')\mhyphen\Tamb
        \to
        (\cO_a,\cO_m)\mhyphen\Tamb.
    \] 
%
\end{proposition}

Since limits are computed objectwise, the forgetful functors all commute with limits. Since the categories of \((\cO_a,\cO_m)\)-Tambara functors are cocomplete diagram categories, we further deduce the existence of left-adjoints.

\begin{proposition}
    The forgetful functors
    \[
	(\cO_{a}',\cO_{m}')\mhyphen\Tamb \to
	(\cO_{a},\cO_{m})\mhyphen\Tamb.
	\]
	has a left-adjoint, the corresponding free functor.
\end{proposition}
\begin{proof}
    Since the forgetful functor commutes with limits, by the adjoint functor theorem, it suffices to show that we have a small set of projective generators. However, when 
    \[
    T=\coprod_{H\subset G} G/H,
    \]
    then the Yoneda lemma shows that the (group completion) of the representable functor \(\cP^G_{\cO_a,\cO_m}(T,\mhyphen)\) is a projective generator.
\end{proof}

Conceptually, these free functors freely adjoin any transfers and norms parameterized by \((\cO_a',\cO_m')\) but not in \((\cO_a,\cO_m)\). 

\begin{example}
    Let \(\cO_a=\cO_m=\cO^{tr}\). An \((\cO_a,\cO_m)\)-Tambara functor here is just a coefficient system of commutative rings. Since \(\cO^{tr}\) is the initial indexing system, for any compatible pair \((\cO_a,\cO_m)\), we have an inclusion
    \[
    (\cO^{tr},\cO^{tr})\subseteq (\cO_a,\cO_m).
    \]
    The corresponding forgetful functor just records the underlying coefficient system of commutative rings. 
    
    For corresponding left-adjoint, it is helpful to factor the inclusion into two steps. The pair \((\cO_a,\cO^{tr})\) is always compatible (see Theorem~\ref{thm:OaOtrCompatible}), we have inclusions of compatible pairs
    \[
    (\cO^{tr},\cO^{tr})\subseteq (\cO_a,\cO^{tr})\subseteq (\cO_a,\cO_m).
    \]
    The left adjoint for first inclusion creates the free \(\cO_a\)-Green functor, putting in all of the missing transfers in algebras and enforcing the Frobenius relation. The left adjoint for the second inclusion then freely puts in the norms.
\end{example}

It is worth noting here that this is the only order that works in general. If \(\cO_m\) is non-trivial, then a consequence of Corollary~\ref{cor:OaAtLeastOm} below is that \((\cO^{tr},\cO_m)\) is never compatible. In other words, we had to put in the missing transfers, and then we can put in the missing norms.

\subsubsection*{Underlying incomplete Mackey functors}

The inclusions here make sense also for more general wide, pullback
stable subcategories like \(\Set^G_{\cong}\), as in
Proposition~\ref{prop:Isos}. Here, we need to also note that Proposition~\ref{prop:ProductsExist} only used the wideness of the categories structuring the norms and transfers, since it used only the identity maps there. Using these more general forgetful functors, we
can talk about ``underlying'' structures.

\begin{example}
Let \(\cO_a\) and \(\cO_m\) be compatible indexing categories. Then for any \((\cO_a,\cO_m)\)-Tambara functor \(\mR\), we have
    \begin{enumerate}
        \item an underlying additive \(\cO_a\)-Mackey functor and
        \item an underlying multiplicative \(\cO_m\)-semi-Mackey functor
    \end{enumerate}
    which arise from the inclusions
    \[
    (\cO_a,\Set^G_{\cong})\subseteq (\cO_a,\cO_m)\supseteq (\Set^G_{\cong},\cO_m).
    \]
\end{example}

The forgetful functor to the underlying additive \(\cO_a\)-Mackey functor also has a left adjoint: this is the bi-incomplete version of the symmetric algebra. The forgetful functor to the underlying multiplicative \(\cO_m\)-semi-Mackey functor is a little stranger, but it also has a left adjoint. This is a bi-incomplete version of Nakaoka's ``Tambarization of a semi-Mackey functor'' \cite[Theorem 2.12]{NakaokaTamb}.

\section{Rewriting compatibility}\label{sec:ReWrite}

The purpose of this section is to provide alternate conditions for
compatibility that are easier to check in practice.
The contravariant functoriality of the pullback gives covariant
functoriality of the dependent product: if \(f\colon S\to T\) and
\(g\colon T\to U\), then we have a natural isomorphism 
\[
\Pi_g\circ \Pi_f\cong \Pi_{g\circ f}\colon \Set^G_{/S}\to\Set^G_{/U}.
\]
We use this to simplify the condition of compatibility: any map in \(\Set^G\) can be written as a disjoint union of composites of fold maps and maps between orbits. Note here that there is also the possibility of some of the disjoint summands being empty. Putting these observations together, it suffices to consider the dependent products along 
\begin{enumerate}
    \item the unique map \(\emptyset \to T\),
    \item a disjoint union of maps \((S_1\to T_1)\amalg (S_2\to T_2)\),
    \item the fold map \(S\amalg S\to S\), and
    \item maps of orbits \(G/H\to G/K\).
\end{enumerate}
Proposition~\ref{prop:Isos} shows that the dependent product along an
isomorphism preserves any pullback stable subcategory, so we will use
this whenever it makes formulae easier. 

\subsection{Main reductions}

We now begin a series of reductions of the condition.

\subsubsection{Initial maps}
\begin{proposition}\label{prop:DepProdMono}
    The dependent product along \(\iota\colon \emptyset\to T\) is always a terminal object \(T\to T\) of the slice category over \(T\).
\end{proposition}
\begin{proof}
    The slice category over the empty sets is the full subcategory of initial objects in \(\Set^G\). Every object in this category is uniquely isomorphic to every other object, and this means that every object is also a terminal object of this category. Being a right adjoint, the dependent product preserves terminal objects.
\end{proof}

\begin{corollary}
    The dependent product along \(\emptyset\to T\) preserves any indexing category.
\end{corollary}



\subsubsection{Disjoint Unions}

\begin{proposition}
    Let \(f\colon T\to T'\) and \(g\colon U\to U'\). If \(h\colon S\to T\amalg U\), let 
    \[
        S_T=h^{-1}(T)
    \]
    let \(h_T:S_T\to T\) be the restriction of \(h\), and similarly for \(S_U\) and \(h_U\). Then we have a natural isomorphism
    \[
    \Pi_{f\amalg g}(h)\cong \Pi_f(h_T)\amalg \Pi_g(h_U).
    \]
\end{proposition}
\begin{proof}
    The assignment
    \[
        h\mapsto (h_T,h_U)
    \]
    gives a functor
    \[
    \Set^G_{/(T\amalg U)}\to\Set^G_{/T}\times\Set^G_{/U}.
    \]
    This is an equivalence of categories, with inverse equivalence given by 
    \[
    \big((S_T\to T),(S_U\to U)\big)\mapsto (S_T\amalg S_U\to T\amalg U).
    \]
    This product decomposition is also natural: we have a commutative diagram
    \[
    \begin{tikzcd}
        {\Set^G_{/(T\amalg U)}}
            \ar[d, "\simeq"']
            &
        {\Set^G_{/(T'\amalg U')}}
            \ar[l, "{(f\amalg g)^\ast}"']
            \\
        {\Set^G_{/T}\times\Set^G_{/U}}
            &
        {\Set^G_{/T'}\times\Set^G_{/U'}}
            \ar[l, "{(f^\ast,g^\ast)}"]
            \ar[u, "\simeq"']
    \end{tikzcd}
    \]
    The result follows from noting that the left vertical map takes \(h\) to \((h_T,h_U)\), the right adjoint to the bottom map is \((\Pi_f,\Pi_g)\), and right vertical map is the disjoint union.
\end{proof}

Being closed under dependent products along disjoint unions follows from simply being closed under disjoint unions and the dependent products along the summands.

\begin{corollary}
    If an indexing system \(\cO\) is closed under dependent products along \(f\colon T\to T'\) and \(g\colon U\to U'\), then it is closed under the dependent product along \(f\amalg g\).
\end{corollary}

\subsubsection{Fold maps}
The dependent product along the fold map is closely connected to the categorical product in the slice categories.

\begin{proposition}
    Let \(\nabla\colon T\amalg T\to T\) be the fold map, and let \(\iota_L,\iota_R\colon T\to T\amalg T\) be the left and right inclusions. We have a natural isomorphism
    \[
    \Pi_{\nabla}\cong \iota_L^{\ast}\timesover{T}\iota_R^{\ast}.
    \]
\end{proposition}
\begin{proof}
    Given any \(S\to T\amalg T\), let 
    \[
        S_L=\iota_L^\ast(S)\text{ and }S_R=\iota_R^\ast(S).
    \]
    We have a natural isomorphism over \(T\amalg T\):
    \[
        S_L\amalg S_R\cong S
    \]
    expressing the disjunctive property of maps to a disjoint union of sets. We now appeal to a direct construction of the dependent product: the fiber over a point \(t\in T\) is the set of sections of \(S\) over \(\nabla^{-1}(t)\). The set \(\nabla^{-1}(t)\) is \(\{t\}\amalg\{t\}\), and a section over this is by construction a pair \((s_L,s_R)\), where \(s_L\in S_L\) and \(s_R\in S_R\) both map to \(t\). This is the same data as a point in the fiber of \(S_L\timesover{T} S_R\) over \(t\).
\end{proof}

\begin{lemma}\label{lem:IndexingSlicesProducts}
    Slices of indexing categories are closed under fiber products.
\end{lemma}
\begin{proof}
    This is a consequence of pullback stability. Let \(f\colon S_1\to T\) and \(g\colon S_2\to T\) both be in \(\cO\). Then the fiber product is defined by the pullback diagram
    \[
        \begin{tikzcd}
            {S_1\timesover{T} S_2}
                \ar[d, "{g'}"']
                \ar[r, "{f'}"]
                &
            {S_2}
                \ar[d, "g"]
                \\
            {S_1}
                \ar[r, "f"]
                &
            {T.}
        \end{tikzcd}
    \]
    Pullback stability of \(\cO\) guarantees that \(g'\) is in \(\cO\) as well, and hence the composite \(f\circ g'\) is.
\end{proof}

\begin{remark}
    Under the equivalences 
    \[
        \Set^G_{/(G/H)}\simeq \Set^H,
    \]
    the fiber product in \(G\)-sets over \(G/H\) is sent to the ordinary product of \(H\)-sets. Lemma~\ref{lem:IndexingSlicesProducts} is then the indexing category version of the statement ``admissible \(H\)-sets are closed under products'', which is \cite[Lemma 4.11]{BHNinfty}.
\end{remark}

\begin{corollary}
    Any indexing category is closed under dependent products along fold maps.
\end{corollary}

We pause here to note that these pieces alone are sufficient to show that the analogue of Green functors always works.

\begin{theorem}\label{thm:OaOtrCompatible}
    For any indexing category \(\cO_a\), the pair \((\cO_a,\cO^{tr})\) is compatible.
\end{theorem}
\begin{proof}
    In \(\cO^{tr}\), the only allowed maps of orbits are isomorphisms, and hence the only conditions we needed to check to ensure compatibility are the first three.
\end{proof}

The \((\cO_a,\cO^{tr})\)-Tambara functors are essentially
\(\cO_a\)-Green functors.  In fact, Strickland's proof for the
additively complete case \(\cO_a=\cO^{gen}\) goes through without
change in the incomplete case. 

\begin{proposition}[{\cite[Proposition 12.11]{Strickland}}]
There is an equivalence of categories between \(\cO\)-Green functors
and \((\cO,\cO^{tr})\)-Tambara functors.
\end{proposition}

\subsection{Admissibility}

Our reductions show that the possible obstruction to compatibility of
an additive and multiplicative indexing category is the dependent
product along a map of orbits 
\[
    G/K\to G/H.
\]
It turns out that this is a surprisingly harsh condition.  By
Proposition~\ref{prop:Isos}, we may assume that \(K\subset H\) and the
map is the canonical quotient. We recall a proposition from
\cite{HMazur}.

\begin{proposition}[{\cite[Proposition 2.3]{HMazur}}]
    Let \(K\subset H\) be subgroups of \(G\), and let \(f\colon T\to G/K\) be a map of \(G\)-sets. Then the dependent product of \(f\) along the canonical quotient \(G/K\to G/H\) is
    \[
    G\timesover{H}\Map^K(H,T_e)\to G/H,
    \]
    where \(T_e=f^{-1}(eK)\) is the \(K\)-set corresponding to \(T\) under the equivalence of categories
    \[
        \Set^K\simeq \Set^G_{/(G/K)}.
    \]
\end{proposition}

This gives us the last piece we need, so we collect all of our reductions into one statement.

\begin{theorem}\label{thm:OaOmCompatibility}
Let \(\cO_a\) and \(\cO_m\) be indexing categories. Then \((\cO_a,\cO_m)\) is compatible if and only if for every pair of subgroups \(K\subset H\) such that \(H/K\) is an admissible \(H\)-set for \(\cO_m\) and for every admissible \(K\)-set \(T\) for \(\cO_a\), the coinduced \(H\)-set \(\Map^K(H,T)\) is admissible for \(\cO_a\).
\end{theorem}

Although this may seem confusing, it records a very conceptual
reformulation. Recall that we have symmetric monoidal Mackey functor
extensions (in the sense of \cite{BohOso} or \cite{HHLocalization}): 
\begin{enumerate}
\item \(\mSet^{\amalg}\) has the coCartesian extension, where categorical transfers are induction, and
\item \(\mSet^{\times}\) has the Cartesian extension, where the categorical transfers are {\coinduction}.
\end{enumerate}

An indexing system \(\cO_a\) is by definition a sub-symmetric monoidal
coefficient system of \(\mSet^{\amalg}\). We have some closure under
induction, but that is not relevant for compatibility. Additionally,
since admissible sets are closed under products, we deduce that
\(\cO_a\) is also a sub-symmetric monoidal coefficient system of the {\emph{Cartesian}} symmetric monoidal Mackey functor 
     \(\mSet^{\times}\).  Theorem~\ref{thm:OaOmCompatibility} is simply
     compatibility with the \(\cO_m\)-Mackey structure. 

\begin{corollary}\label{cor:OaOmSymMonoidal}
A pair \((\cO_a,\cO_m)\) is compatible if and only if \(\cO_a\) is
actually a sub-symmetric monoidal \(\cO_m\)-Mackey functor of the
Cartesian monoidal Mackey functor \(\mSet^\times\).
\end{corollary}

\section{Limits on Compatibility}\label{sec:Limits}

For a general pair $(\cO_a,\cO_m)$ of indexing categories, it can be 
difficult to check that $\cO_m$ distributes over $\cO_a$.  We make
a few basic observations here.  

\subsection{The additive hull}

It is straightforward to check that compatibility conditions are
preserved by intersection.

\begin{lemma}\label{lem:CompatibleHull}
    If \((\cO_a,\cO_m)\) and \((\cO_a',\cO_m')\) are compatible pairs, then
    \[
        (\cO_a\cap\cO_a',\cO_m\cap\cO_m')
    \]
    is a compatible pair.
\end{lemma}

From this, we deduce that there is always a kind of ``compatible hull'' of the additive indexing category making a compatible pair.

\begin{proposition}\label{prop:OaHull}
    For any pair of indexing categories \(\cO_a\) and \(\cO_m\), there is a minimal \(\overline{\cO}_a\) containing \(\cO_a\) such that \((\overline{\cO}_a,\cO_m)\) is compatible.
\end{proposition}
\begin{proof}
    Consider the set
    \[
    \mathcal E=\big\{(\cO_a',\cO_m')\mid (\cO_a',\cO_m')\text{ compatible \& } (\cO_a,\cO_m)\leq (\cO_a',\cO_m')\big\},
    \]
    and let 
    \[
    \overline{\cO}_a=\bigcap_{(\cO_a',\cO_m')\in\mathcal E}\cO_a'\text{ and }\overline{\cO}_m=\bigcap_{(\cO_a',\cO_m')\in\mathcal E}\cO_m'.
    \]
    The set \(\mathcal E\) non-empty because \((\cO^{gen},\cO_m)\) is always in this, and this also shows that \(\overline{\cO}_m=\cO_m\). By Lemma~\ref{lem:CompatibleHull}, the pair \((\overline{\cO}_a,\cO_m)\) is compatible.
\end{proof}

\begin{proposition}\label{prop:OmIsometryLike}
    If \((\cO_a,\cO_m)\) is a compatible pair and \(H/K\) is an \(\cO_m\)-admissible set for \(H\), then for every \(L\subset H\) such that \(K\) is sub-conjugate to \(L\), the \(H\)-set \(H/L\) is \(\cO_a\)-admissible.
\end{proposition}
\begin{proof}
    For any subgroup \(K\), the \(K\)-set
    \[
        \{a,b\}:=\ast\amalg\ast
    \]
    is always admissible for any indexing system. By Theorem~\ref{thm:OaOmCompatibility}, if \(H/K\) is an \(\cO_m\)-admissible \(H\)-set, then we must have that
    \[
    \Map^K\big(H,\{a,b\}\big)\cong\Map\big(H/K,\{a,b\}\big)\in\cO_a(H).
    \]
    
    Since admissible sets are closed under conjugation, it suffices to consider the case that \(L\) contains \(K\). Consider the function 
    \[
    f\colon H/K\to\{a,b\}
    \]
    defined by
    \[
    f(hK)=\mycases{
    a & h\in L \\
    b & h\not\in L.
    }
    \]
    Then the stabilizer of \(f\) is \(L\), and hence we have a summand 
    \[
    H/L\cong H\cdot f\subset\Map^K\big(H,\{a,b\}\big),
    \]
    which means \(H/L\) is admissible for \(\cO_a\), as desired.
\end{proof}

\begin{corollary}\label{cor:OaAtLeastOm}
    If \((\cO_a,\cO_m)\) is a compatible pair, then \(\cO_a\geq \cO_m\) in the partial order on indexing categories given by inclusion.
\end{corollary}

\begin{corollary}\label{cor:AbsNormsAreComplete}
    If \(G\) is an admissible \(G\)-set for \(\cO_m\), then \(\cO_a=\Set^G\) is the terminal indexing category.
\end{corollary}

Put another way, having the norm from the trivial subgroup to \(G\) can only happen for ordinary incomplete Tambara functors.

Corollary~\ref{cor:OaAtLeastOm} also bounds sharply the number of compatible pairs. 

\begin{corollary}
Let \(n(G)\) be the cardinality of the poset of indexing categories for \(G\) and let \(c(G)\) be the number of pairs of indexing categories \((\cO,\cO')\) such that \(\cO\geq \cO'\). Then we have
\[
\frac{\#\text{ compatible pairs}}{\#\text{ all pairs}}\leq \frac{c(G)}{n(G)^2} \leq \frac{1}{2}+\frac{1}{2n(G)}.
\]
\end{corollary}
\begin{proof}
    The first bound follows immediately from Corollary~\ref{cor:OaAtLeastOm} and the observation that there are \(n(G)^2\) total possible pairs. The second bound follows from a combinatorial observation. If we consider all poset structures on the set of \(n\) elements, then the maximal number of pairs \((a,b)\) with \(a\geq b\) occurs when the poset is a total order. In this case, we have \(\tfrac{1}{2}(n^2+n)\), from which the second bound follows.
\end{proof}

\begin{example}
    For \(G=C_p\), the poset of indexing categories is the total order on two elements, and we achieve the bound: \(3/4\) of the pairs are compatible.
\end{example}

For larger groups, we expect the bounds to generically less than \(1/2\), as the poset of indexing categories seem generically to contain many incomparable elements.

\begin{example}
    For \(G=C_{p^2}\), Balchin--Barnes--Roitzheim showed there are \(5\) indexing categories in the poset, and looking at them shows that there are \(13\) pairs \((\cO,\cO')\) with \(\cO\geq \cO'\) \cite[Theorem 2]{BBRAssoc}. Of these, \(12\) are compatible: the indexing category for the little disks operad on \(\infty(1+\lambda)\), where \(\lambda\) is a \(2\)-dimensional faithful representation of \(C_{p^2}\), is not compatible with itself, due to Corollary~\ref{cor:AbsNormsAreComplete}.
\end{example}

Using the classifications of Balchin--Barnes--Roitzheim, of Balchin--Bearup--Pech--Roitzheim, and of Rubin, we can at least provide upper bounds on the number of compatible pairs.

\begin{example}
    For \(G=C_{p^3}\), Balchin--Barnes--and Roitzheim showed there are \(14\) indexing categories \cite[Theorem 2]{BBRAssoc}, and looking at the poset structure (as depicted in \cite[Section 3.2]{RubinSteinLin}), we see \(67\) are comparable. Of these, \(55\) conform to the conditions of Proposition~\ref{prop:OmIsometryLike}.
\end{example}

\begin{example}
    For \(G=C_{pq}\) with \(p\) and \(q\) distinct primes, Balchin--Bearup--Pech--Roitzheim and Rubin showed there are \(10\) indexing categories \cite[Section 3]{BBPR}, \cite[Section 3.2]{RubinSteinLin}. Of the \(100\) possible pairs, \(44\) are comparable, and \(39\) conform to the conditions of Proposition~\ref{prop:OmIsometryLike}.
\end{example}

To try to get sharper estimates on the number of compatible pairs, we first look at the edge-cases: when an indexing system is compatible with itself. Recall that an indexing category is ``linear isometry-like'' if whenever \(H/K\) is an admissible \(H\)-set, the sets \(H/L\) for any \(K\subset L\subset H\) are also all admissible \cite{RubinSteinLin}. 

\begin{corollary}
    If \((\cO,\cO)\) is compatible, then \(\cO\) is linear isometry-like.
\end{corollary}

These conditions alone already put stringent constraints, and we have only used the ``trivial'' part of \(\cO_a\). Any additional stabilizer types will give more conditions. 

\begin{remark}\label{rem:CantGrowOa}
    It is not necessarily the case that if \((\cO_a,\cO_m)\) is
    compatible and \(\cO_a'\geq \cO_a\), then \((\cO_a',\cO_m)\) is compatible: additional stabilizers can show up.
\end{remark}

\subsection{Examples: little disks and linear isometries}

The initial motivation for our study of \(\Ninfty\)-operads was the
problem of understanding to what extent linear isometries and little
disks operads failed to be equivalent, equivariantly.  We recall the
explicit conditions for admissibility for linear isometries and little
disks operads.

\begin{theorem}[{\cite[Theorems 4.18 \& 4.19]{BHNinfty}}]\label{thm:Admiss}
    Let \(U\) be a universe. 
    
    A finite \(H\)-set \(T\) is admissible for the linear isometries operad for \(U\) if and only if there is an \(H\)-equivariant isometry
    \[
        \mathbb R\cdot T\otimes U\hookrightarrow U,
    \]
    where \(\mathbb R\cdot T\) is the permutation representation generated by \(T\).
    
    A finite \(H\)-set \(T\) is admissible for the little disks operad for \(U\) if and only if there is an \(H\)-equivariant embedding
    \[
    T\hookrightarrow U.
    \]
\end{theorem}

\begin{proposition}\label{prop:DiskNoGo}
    Let \(V\) be a faithful representation of \(G\). If there exists a subgroup \(H\) such that \(G/H\) does not embed in \(\infty(1+V)\), then the indexing category \(\cO\) associated to the little disks operad is not compatible with itself.
\end{proposition}
\begin{proof}
    Since \(V\) is a faithful representation of \(G\), for any
    non-trivial subgroup \(H\subset G\), the \(H\)-fixed points of
    \(V\) are a proper subspace of \(V\). Since there are finitely
    many non-trivial subgroups, we find that the collection of all
    vectors in \(V\) with a non-trivial stabilizer is a finite union
    of hyperplanes and hence a proper subset. We therefore have
    vectors with a trivial stabilizer. Any of these gives an
    equivariant embedding \(G\hookrightarrow V\), and Theorem~\ref{thm:Admiss} then says that \(G\) is an admissible \(G\)-set for \(\cO\).
    
    Corollary~\ref{cor:AbsNormsAreComplete} shows that if \(G\) is an admissible \(\cO_m\) set, then all \(H\)-sets are \(\cO_a\) admissible for all \(H\). However, Theorem~\ref{thm:Admiss} also shows that the assumption that \(G/H\) does not embed is equivalent to the assertion that \(G/H\) is not admissible. This means \(\cO\) cannot be compatible with itself.
\end{proof}

One key result in this direction was \cite[Theorem 4.24]{BHNinfty}, which showed the existence of little disks operads inequivalent to any linear isometries operad, provided the group is not simple. The key step was producing a representation of \(G\) such that \(G\) embeds but \(G/N\) does not for some normal subgroup: inducing up the reduced regular representation for \(N\) works.

\begin{corollary}\label{cor:SelfIncompatible}
    For any non-simple group \(G\), there is a representation \(V\) such that the indexing category for the little disks on \(V\) is not compatible with itself.
\end{corollary}

On the other hand, as topology informs us, the indexing category for the linear isometries operad on a universe \(U\) is always compatible with its corresponding little disks. We can see this algebraically.

\begin{proposition}\label{prop:DisksIsomWorks}
Let $U$ be a universe for $G$, let $\cO_a$ be the indexing category associated to the little disks operad for $U$ and let $\cO_m$ be the indexing category associated to the linear isometries operad for $U$. Then \((\cO_a,\cO_m)\) is compatible.
\end{proposition}
\begin{proof}
    By Theorem~\ref{thm:OaOmCompatibility}, it suffices to show that the \(\cO_a\)-admissible sets are closed under {\coinduction} along \(\cO_m\)-admissible sets. 
    
    When \(T=H/K\), the isometric embedding condition of Theorem~\ref{thm:Admiss} can be rewritten as the existence of an isometric embedding
    \[
    \mathbb R[H]\otimesover{\mathbb R[K]}U\hookrightarrow U.
    \]
    Using the isomorphism of induction with {\coinduction} in representations, we deduce our desired result. If \(T\) is a finite \(K\)-set that equivariantly embeds into \(i_K^\ast U\), then a choice of such an embedding gives an embedding
    \[
    \Map^K(H,T)\hookrightarrow\Map^K(H,i_K^\ast U)\cong \mathbb R[H]\otimesover{\mathbb R[K]}U\hookrightarrow U,
    \]
    as desired.
\end{proof}

\subsection{The multiplicative hull}
In Proposition~\ref{prop:OaHull}, we saw that there is a {\emph{smallest}} additive \(\cO_a\) compatible with any fixed multiplicative \(\cO_m\) using the intersection. The poset of indexing categories is actually a lattice, and intersection realizes the ``meet''. The join operation is more difficult to describe, but it has been identified by Rubin \cite{RubinOpLifts}. Rubin works with ``transfer systems'' or ``norm systems'', another equivalent form of the data of an indexing system defined independently by Rubin \cite{RubinSteinLin} and Balchin--Barnes--Roitzheim \cite{BBRAssoc}. For convenience, we state the result in indexing categories.

\begin{proposition}[{\cite[Proposition 3.1]{RubinOpLifts}}]\label{prop:Join}
    The join of two indexing categories \(\cO\) and \(\cO'\) is just the finite coproduct complete subcategory generated by them.
\end{proposition}
In other words, pullback stability is automatic. What this means for us, however, is that we can join together multiplicatively compatible indexing categories.

\begin{proposition}\label{prop:OmHull}
    Let \(\cO_a\) be an indexing category. If \((\cO_a,\cO_m)\) and \((\cO_a,\cO_m')\) are both compatible, then \(\cO_a\) is complatible with the join \(\cO_m\vee\cO_m'\).
\end{proposition}
\begin{proof}
    One way to reinterpret Proposition~\ref{prop:Join} is that a map of orbits \(G/H\to G/K\) is in the join if and only if it can be written as a composition of maps 
    \[
    G/H\xrightarrow{f_0} G/H_1\xrightarrow{f_1}\dots\xrightarrow{f_n} G/K,
    \]
    with the maps \(f_i\) for \(0\leq i\leq n\) in at least one of \(\cO_m\) or \(\cO_m'\). We assumed that \(\cO_a\) was compatible with both \(\cO_m\) and \(\cO_m'\), and hence it is closed under dependent products along any of the maps in them. The proof of Theorem~\ref{thm:OaOmCompatibility} shows it suffices to check closure on maps between orbits, so we are done.
\end{proof}

This means we can find a {\emph{largest}} multiplicative \(\cO_m\)
compatible with a given \(\cO_a\). 

\begin{corollary}
    For any \(\cO_a\), there is a largest \(\cO_m\) such that \((\cO_a,\cO_m)\) is compatible.
\end{corollary}
\begin{proof}
    We simply join together all \(\cO_m\) such that \((\cO_a,\cO_m)\) is compatible. The set of these is non-empty, because \(\cO^{tr}\) is compatible with all \(\cO_a\).
\end{proof}

\begin{remark}
    In contrast to Remark~\ref{rem:CantGrowOa}, if \((\cO_a,\cO_m)\) is compatible and \(\cO_m'\leq \cO_m\), then \((\cO_a,\cO_m')\) is compatible. This reinforces an underlying theme that \(\cO_m\) puts constraints on \(\cO_a\), but not vice versa.
\end{remark}

\section{Change of group}\label{sec:Dragons}

A basic property of Mackey functors that makes computation easier is
the fact that induction is naturally isomorphic to {\coinduction}.
This is a reflection of the Wirthmuller isomorphism of ``genuine''
equivariant stable homotopy theory, and it makes it relatively easy to
form various kinds of resolutions we might want for homological
algebra. We can view this as the \(G\)-equivariant version of
``additive'', since it means that finite sums indexed by the group are
finite products indexed by the group.  Work of Berman cleanly explains
this point~\cite{Berman}.

In incomplete Tambara functors, we have none of this. The coproduct
and product do not agree, and while the forgetful functor to
incomplete Tambara functors for a subgroup has both adjoints, they
never agree. The coproduct, product, and right adjoints are all
expressible as corresponding functors on the underlying Mackey
functors (the latter ones because the forgetful functor commutes with
limits). In certain cases, the left adjoint to the forgetful functor
is also expressible in terms of a functor on Mackey functors: the
norm. We close by giving a few easy results of this form for
bi-incomplete Tambara functors, and then stating some conjectures for
structure that would tie everything together.

\subsection{Restriction functors \& {\coinduction}}
One of the key technical lemmas used in studying incomplete Tambara functors was knowing when a pair of adjoint functors on a category extends to a pair of adjoint functors on polynomials in that category with some restricted class of norms (\cite[Theorem 2.17]{BHOTamb}). The main ingredients were 
\begin{enumerate}
    \item the image of induction forms an ``essential sieve'': if \(f\colon S\to G\timesover{H} T\) is a map, then this is isomorphic over \(G\timesover{H}T\) to a map of the form 
    \[
    G\timesover{H}f'\colon G\timesover{H}S'\to G\timesover{H}T,
    \]
    and 
    \item an \(H\)-equivariant map \(f\) is in \(i_H^\ast\cO\) if and only if \(G\timesover{H}f\) is in \(\cO\).
\end{enumerate}  
At no point in the proof did we use the fact that all transfers exist,
so we deduce that the result goes through in this setting without change.

\begin{proposition}
    Precomposition with induction and restriction, respectively give an adjoint pair of functors
    \[
        i_H^\ast\colon (\cO_a,\cO_m)\mhyphen\Tamb^G\rightleftarrows (i_H^\ast\cO_a,i_H^\ast\cO_m)\mhyphen\Tamb^H\colon \CoInd_{H}^G.
    \]
\end{proposition}

The key feature here is that {\coinduction} and restriction are ``the same'' functor no matter which indexing categories we use: simply precompose with restriction or induction respectively. In particular, the values are determined when we forget all the way down to coefficient systems.

\subsection{Induction}

Since the restriction functors commute with limits, for formal reasons
we know they have left adjoints. As in Mackey functors and incomplete
Tambara functors, these left adjoints are concisely described as left
Kan extensions.

\begin{definition}
    Let
    \[
        n_H^G\colon (i_H^\ast\cO_a,i_H^\ast\cO_m)\mhyphen\Tamb^H\to (\cO_a,\cO_m)\mhyphen\Tamb^G
    \]
    be the left Kan extension along the induction functor
    \[
        G\timesover{H}(\mhyphen)\colon \cP^H_{i_H^\ast\cO_a,i_H^\ast\cO_m}\to \cP^G_{\cO_a,\cO_m}.
    \]
\end{definition}

Since induction is product preserving, the left Kan extension along it
preserves product preserving functors~\cite{KellyLack}. It is formal
that this gives the left adjoint.

\begin{proposition}
The functor \(n_H^G\) is left-adjoint to the restriction functor.
\end{proposition}

In Mackey functors, this left adjoint is actually isomorphic to the right adjoint. In \(\cO_a\)-Mackey functors, this is not always the case. Instead, we have such an isomorphism when \(G/H\) is admissible (this is \(\m\pi_0\) of \cite[Theorem 3.25]{BHOSpectra}).

\subsection{The norm \& externalized forms}

The key application of the theses of Mazur and of Hoyer was that the
norm functor on Mackey functors describes the left adjoint on Tambara
functors~\cite{MazurThesis, HoyerThesis}, and hence we have a
reinterpretation of Tambara functors as the \(G\)-commutative monoids
in the category of Mackey functors~\cite{HHLocalization}.  A similar
statement here would provide a clean, algebraic interpretation of
Corollary~\ref{cor:OaOmSymMonoidal}.  We present several conjectures
here. 

\begin{conjecture}\label{conj:NormsCompat}
    If \((\cO_a,\cO_m)\) is a compatible pair, then for every admissible \(H/K\) for \(\cO_m\), {\coinduction} restricts to a functor
    \[
        \Map^K(H,\mhyphen)\colon i_K^\ast\cO_a\to i_H^\ast\cO_a.
    \]
In this case, we will also say ``{\coinduction} preserves \(\cO_a\)''.
\end{conjecture}

Since these are wide subcategories, this conjecture is really a statement about the maps in the category being closed under {\coinduction}. This is a twisted version of Lemma~\ref{lem:ClosureUnderProducts}, and it gives the twisted version of Corollary~\ref{cor:AOSymMonoid}.

\begin{proposition}
    If {\coinduction} from \(K\) to \(H\) preserves \(\cO\), then it induces a functor
    \[
        \cA^K_{i_K^\ast\cO}\to \cA^{H}_{i_H^\ast\cO}.
    \]
\end{proposition}

This is the heart of generalizing the Hoyer--Mazur norm, and it gives
us the following definition which we conjecture is correct.

\begin{definition}\label{def:Norms}
    If {\coinduction} from \(K\) to \(H\) preserves \(\cO\), then left Kan extension gives a norm functor
    \[
    {}_{\cO}N_K^H\colon\cO\mhyphen\Mackey^K\to\cO\mhyphen\Mackey^H
    \]
    that commutes with the box product.
\end{definition}

Definition~\ref{def:Norms} says that the category of \(\cO_a\)-Mackey functors is naturally a symmetric monoidal \(\cO_m\)-Mackey functor under the box product and norms: the restriction maps are the ordinary restrictions (which are strong symmetric monoidal) and the transfer maps are the norms. 

\begin{warning}
    Just as the box product depends heavily on \(\cO\), ranging from the levelwise tensor product (for \(\cO^{tr}\)) to the usual box product on Mackey functors, so too will any norms. Different \(\cO\) will give different definitions of the norm.
\end{warning}

The symmetric monoidal \(\cO_m\)-Mackey structure gives a collection of endofunctors of \(i_H^\ast\cO_a\)-Mackey functors for all \(H\): for any admissible \(H/K\) for \(\cO_m\), let
\[
{}_{\cO_a}N^{H/K}:={}_{\cO_a} N_K^Hi_K^\ast.
\]
This functor is isomorphic to the left Kan extension along \(\Map(H/K,\mhyphen)\) \cite{WittGreen}.
We should think of this as the categorical version of the formula expressing the action of the Burnside ring of finite \(H\)-sets on any Mackey functor \(\mM\) evaluated at \(G/H\):
\[
[H/K]\cdot m=tr_{K}^H res_K^H(m).
\]
Multiplying together the norm functors gives us an endofunctor for any \(\cO_m\)-admissible \(H\)-set \(T\): writing \(T\) as \(T=H/K_1\amalg\dots\amalg H/K_n\), let
\[
{}_{\cO_a}N^T(\mM):= {}_{\cO_a}N^{H/K_1}\mM\Box\dots\Box {}_{\cO_a} N^{H/K_n}\mM.
\]

With this, we conjecture the general external form of a \((\cO_a,\cO_m)\)-Tambara functor.

\begin{conjecture}\label{conj:OmCommMonoids}
A \((\cO_a,\cO_m)\)-Tambara functor is a \(\cO_m\)-commutative
monoid in the symmetric monoidal \(\cO_m\)-Mackey functor of
\(\cO_a\)-Mackey functors. 
\end{conjecture}

Unpacking this, this would mean that a \((\cO_a,\cO_m)\)-Tambara functor is the following data:
\begin{enumerate}
    \item A \(\cO_a\)-Mackey functor \(\mR\)
    \item a unital commutative monoid structure on \(\mR\): \(\mR\Box\mR\to\mR\), and
    \item for every map of \(\cO_m\)-admissible \(H\) sets \(S\to T\), an ``external norm'': a map of commutative monoids
    \[
    {}_{\cO_a}N^S i_H^\ast\mR\to {}_{\cO_a}N^T i_H^\ast\mR.
    \]
\end{enumerate}
These are required to be compatible in the sense that the norm map for \(S\to T\) restricts to the norm map for the restriction of \(S\to T\) and the norm maps compose to give the norm map for the composite.

\printbibliography



\end{document}